\documentclass[10pt,reqno]{amsart} 
\usepackage{amssymb}
\usepackage[cp1250]{inputenc}   
\usepackage{hyperref}
\usepackage[all]{xy}            
\usepackage{verbatim}
\usepackage{tikz}
    \usetikzlibrary{shadows,matrix,arrows}
\usepackage{mathtools}
\usepackage{longtable}

\newtheorem{Thm}{Theorem}[]
\newtheorem{Prp}[Thm]{Proposition}

\theoremstyle{definition}

\theoremstyle{remark}
\newtheorem{Rmk}[Thm]{Remark}
\numberwithin{equation}{section}

\def\id{\mathrm{id}}        
\def\opp{\mathrm{op}}       
\def\Ker{\mathrm{Ker}}
\def\Im{\mathrm{Im}}
\def\Hom{\mathrm{Hom}}

\def\N{\mathbb{N}}
\def\Z{\mathbb{Z}}

\def\N{\mathbb{N}}

\def\chr{\mathrm{char}}
\def\bbinom#1#2{\ensuremath{\left(\kern-.3em\left(\genfrac{}{}{0pt}{}{\!#1\!}{\!#2\!}\right)\kern-.3em\right)}}
\newcommand{\icases}[7]{#7\{\!\begin{smallmatrix}
                                #1\hfill;  & ~#5#2\hfill &\!\!\\[#6]
                                #3\hfill;  & ~#5#4\hfill &\!\!
                              \end{smallmatrix}}
\newcommand{\iicases}[9]{#9\{\!\begin{smallmatrix}
                                #1\hfill;  & ~#7#2\hfill &\!\!\\[#8]
                                #3\hfill;  & ~#7#4\hfill &\!\!\\[#8]
                                #5\hfill;  & ~#7#6\hfill &\!\!
                              \end{smallmatrix}}     

\begin{document}
\title[Hochschild (Co)Homology of Exterior Algebras using AMT]{Hochschild (Co)Homology of Exterior Algebras\\ using Algebraic Morse Theory}
\author{Leon Lampret}
\address{Department of Mathematics, University of Ljubljana, Slovenia}
\email{leon.lampret@imfm.si}
\author{Aleš Vavpetič}
\address{Department of Mathematics, University of Ljubljana, Slovenia}
\email{ales.vavpetic@fmf.uni-lj.si}
\date{December 28, 2015}

\begin{abstract}
In \cite{citearticleHanXuHHEA}, the authors computed the additive and multiplicative structure of $H\!H^\ast\!(A;\!A)$, where $A$ is the $n$-th exterior algebra over a field. In this paper, we derive all their results using a different method (AMT), as well as calculate the additive structure of $H\!H_k\!(A;\!A)$ and $H\!H^k\!(A;\!A)$ over $\Z$. We provide concise presentations of algebras $H\!H_\ast\!(A;\!A)$ and $H\!H^\ast\!(A;\!A)$, as well as determine their generators in the Hochschild complex. Lastly, we compute an explicit free resolution (spanned by multisets) of the $A^e$-module $A$ and describe the homotopy equivalence to its bar resolution.
\end{abstract}

\maketitle


\subsection*{Conventions} Throughout this article, $R$ will denote a commutative unital ring, $A\!=\!\Lambda[x_1,\ldots,x_n] \!=\!\Lambda(R^n)$ will be the $n$-th exterior algebra, and $A^e\!=\!A\!\otimes_{\!R}\!A^\opp$ its enveloping algebra, so that $A$-$A$-bimodules correspond to $A^e$-modules.
\par Symbols $\sigma$ and $\tau$ will denote a set and a multiset respectively, or equivalently a strictly increasing and an increasing sequence, of elements from $[n]\!=\!\{1,\ldots,n\}$.  Let $\binom{[n]}{k}\!=\! \{k\text{-element subsets of }[n]\}$ and $\bbinom{\smash{[n]}}{k}\!=\! \{k\text{-element multisubsets of }[n]\}$. Then we have $\big|\!\binom{[n]}{k}\!\big|\!=\!\binom{n}{k}\!=\! \frac{n!}{k!(n-k)!}$ and $\big|\!\!\bbinom{\smash{[n]}}{k}\!\!\big|\!=\! \bbinom{n}{k}\!=\!\binom{n+k-1}{k}$.
\par For the sake of brevity, we omit the wedge $\wedge$ and product $\cdot$ symbols. We denote $x_\sigma\!=\!\bigwedge_{i\in\sigma}\!x_i$ and $x_\tau\!=\!\prod_{i\in\tau}\!x_i$, so that algebras $\Lambda[x_1,\ldots,x_n]$ and $R[x_1,\ldots,x_n]$ have $R$-module bases $\big\{x_\sigma;\: k\!\in\!\N, \sigma\!\in\!\binom{[n]}{k}\!\big\}$ and $\big\{x_\tau;\: k\!\in\!\N, \tau\!\in\!\bbinom{\smash{[n]}}{k}\!\!\big\}$ respectively.

\vspace{4mm}
\section{Bar resolution} Using algebraic Morse theory, we find an explicit homotopy equivalence between the bar resolution of bimodule $A$ and a minimal free resolution.

\subsection{Complex}\label{identifications} The \emph{bar resolution} of any associative unital $R$-algebra $A$ is
$$\begin{array}{c}
B_\ast\!:\;\; \ldots\longrightarrow A^{\otimes k+2} \longrightarrow A^{\otimes k+1} \longrightarrow \ldots\longrightarrow A^{\otimes 2} \longrightarrow A \longrightarrow 0,\\
b_k\!: a_0\!\otimes\!\cdots\!\otimes\!a_{k+1} \longmapsto \sum_{0\leq i\leq k} (-1)^i a_0\!\otimes\!\cdots\!\otimes\!a_ia_{i+1}\!\otimes\!\cdots\!\otimes\!a_{k+1},
\end{array}$$
which is an exact complex of $A^e$-modules, where $\otimes\!=\!\otimes_{\!R}$ and $A^e$ acts on $A^{\otimes k+2}$ by \[(\alpha\!\otimes\!\beta)(a_0\!\otimes\!a_1\!\otimes\!\cdots\!\otimes\!a_k\!\otimes\!a_{k+1})=
(\alpha a_0)\!\otimes\!a_1\!\otimes\!\cdots\!\otimes\!a_k\!\otimes\!(a_{k+1}\beta).\] Hochschild (co)homology of $A$ with coefficients in an $A^e$-module $M$ is the homology of the complex $(C_\ast,\partial_\ast)\!=\! M\!\otimes_{\!A^{\!e}}\!B_\ast$ and $(C^\ast,\delta^\ast)\!=\! \Hom_{A^{\!e}}(B_\ast,M)$. Note that\vspace{-1mm}
\[M\otimes_{\!A^{\!e}}A^{\otimes k+2} \cong M\otimes_{\!R}A^{\otimes k}\vspace{-1mm}\] via
$m\otimes(a_0\!\otimes\!a_1\!\otimes\!\ldots\!\otimes\!a_k\!\otimes\!a_{k+1}) \longmapsto (a_0ma_{k+1})\otimes(a_1\!\otimes\!\ldots\!\otimes\!a_k)$ and\vspace{-1mm}
\[\Hom_{A^{\!e}}(A^{\otimes k+2}\!,M)\cong \Hom_R(A^{\otimes k}\!,M)\vspace{-1mm}\] via $\varphi\longmapsto(a_1\!\otimes\!\ldots\!\otimes\!a_k\!\mapsto\!\varphi(1\!\otimes\!a_1\!\otimes\!\ldots\!\otimes\!a_k\!\otimes\!1))$.

\subsection{AMT} To a chain complex of free modules $(B_\ast,b_\ast)$ we associate a weighted digraph $\Gamma_{B_\ast}$ (vertices are basis elements of $B_\ast$, weights of edges are nonzero entries of matrices $b_\ast$). Then we carefully select a matching $\mathcal{M}$ in this digraph, so that its edges have invertible weights and if we reverse the direction of every $e\!\in\!\mathcal{M}$ in $\Gamma_{B_\ast}$, the obtained digraph $\Gamma^{\mathcal{M}}_{B_\ast}$ contains no directed cycles and no infinite paths in two adjacent degrees. Under these conditions, the AMT theorem  (Sk\"{o}ldberg \cite{citearticleSkoldbergMTFAV}, Welker \& J\"{o}llenbeck \cite{citearticleJollenbeckADMTACA}) provides a homotopy equivalent complex $(\mathring{B}_\ast,\mathring{b}_\ast)$, spanned by the unmatched vertices in $\Gamma^{\mathcal{M}}_{B_\ast}$, and with the boundary $\mathring{b}_\ast$ given by the sum of weights of directed paths in $\Gamma^{\mathcal{M}}_{B_\ast}$. For more details, we refer the reader to the two articles above (which specify the homotopy equivalence), or \cite{citeLampretVavpeticCLAAMT} for a quick formulation.

\subsection{Digraph} Since $R$-module $A$ is free on $\{x_\sigma;\, \sigma\!\subseteq\![n]\}$, the $A^{\!e}$-module $A^{\otimes k+2}$ is free on $\{1\!\otimes\!x_{\sigma_1}\!\!\otimes\!\ldots\!\otimes\!x_{\sigma_k}\!\!\otimes\!1;\: \sigma_1,\ldots,\sigma_k\!\subseteq\![n]\}$. These tensors are the vertices of the digraph $\Gamma_{B_\ast}$. From the definition of $b_\ast$, we see that the edges of $\Gamma_{B_\ast}$ and their weights are of three different forms, namely\vspace{-1mm}
\[\xymatrix@R=9.0mm@C=3mm{&1\!\otimes\!x_{\sigma_1}\!\!\otimes\!\ldots\!\otimes\!x_{\sigma_k}\!\!\otimes\!1
\ar[dl]_{x_{\sigma_{\!1}}\!\otimes1}\ar[d]|-{(-1)^i}\ar[dr]^{(-1)^k1\otimes x_{\sigma_{\!k}}}&\\
                          1\!\otimes\!x_{\sigma_2}\!\!\otimes\!\ldots\!\otimes\!x_{\sigma_k}\!\!\otimes\!1&
                          1\!\otimes\!x_{\sigma_1}\!\!\otimes\!\ldots\!\otimes\!x_{\sigma_i}x_{\sigma_{i+1}}\!\!\otimes\!\ldots\!\otimes\!x_{\sigma_k}\!\!\otimes\!1&
                          1\!\otimes\!x_{\sigma_1}\!\!\otimes\!\ldots\!\otimes\!x_{\sigma_{k\!-\!1}}\!\!\otimes\!1,\vspace{-1mm}}\]
where $x_{\sigma_i}x_{\sigma_{i+1}}\!=\! \icases{(-1)^jx_{\sigma_{\!i}\cup\sigma_{\!i+1}}}{\sigma_i\cap\sigma_{i+1}=\emptyset}{0}{\text{otherwise}}{\!\!}{-1pt}{\Big}$ and $j$ is the number of transpositions needed to transform the concatenated elements of $\sigma_i,\sigma_{i+1}$ into an increasing order. By the normalized resolution \cite[1.1.14]{citeLodayCH}, we may assume all $\sigma_1,\ldots,\sigma_k$ are nonempty.

\subsection{Matching} Let us explain how we construct the matching. Given a vertex $1\!\otimes\!x_{\!\sigma_{\!1}}\!\!\otimes\!\ldots\!\otimes\!x_{\!\sigma_{\!k}}\!\!\otimes\!1$, going up means splitting some $\sigma_i$ into $\sigma'_i$ and $\sigma_i\!\setminus\!\sigma'_i$. The simplest choice is $i\!=\!1$ and $\sigma'_i\!=\!\{\max\sigma_i\}$, i.e. let $\mathcal{M}\!=\!\Big\{ \begin{smallmatrix}1\otimes x_i\otimes x_{\sigma_{\!1}\!\setminus\!\{i\}}\otimes \ldots\otimes x_{\sigma_{\!k}} \!\otimes1\\[-2pt] \hspace{-50pt}\downarrow\\[-2pt]
                   1\otimes x_{\sigma_{\!1}}\!\otimes\ldots\otimes x_{\sigma_{\!k}}\!\otimes1\hspace{25pt}\end{smallmatrix}\!; i\!=\!\max\sigma_1\Big\}$. Then the unmatched vertices are
$\mathcal{\mathring{M}}\!=\!\big\{ 1\!\otimes\!x_{i_1}\!\!\otimes\!x_{\sigma_{\!2}}\!\!\otimes\!\ldots\!\otimes\!x_{\sigma_{\!k}}\!\!\otimes\!1;\: i_1\!\leq\!\max\sigma_2\big\}$. If we add edges $\Big\{ \begin{smallmatrix}1\otimes x_{i_1}\!\otimes x_i\otimes x_{\sigma_{\!2}\!\setminus\!\{i\}}\otimes \ldots\otimes x_{\sigma_{\!k}} \!\otimes1\\[-2pt] \hspace{-40pt}\downarrow\\[-2pt]
1\otimes x_{i_1}\!\otimes x_{\sigma_{\!2}}\!\otimes\ldots\otimes x_{\sigma_{\!k}}\!\otimes1\hspace{25pt}\end{smallmatrix}\!; i\!=\!\max\sigma_2\Big\}$ to $\mathcal{M}$, then the
unmatched vertices are $\mathcal{\mathring{M}}\!=\!\big\{ 1\!\otimes\!x_{i_1}\!\!\otimes\!x_{i_2}\!\!\otimes\!x_{\sigma_{\!3}}\!\!\otimes\!\ldots\!\otimes\!x_{\sigma_{\!k}}\!\!\otimes\!1;\: i_1\!\leq\!i_2\!\leq\!\max\sigma_3\big\}$. If we add edges $\Big\{ \begin{smallmatrix}
1\otimes x_{i_1}\!\otimes x_{i_2}\!\otimes x_i\otimes x_{\sigma_{\!3}\!\setminus\!\{i\}}\otimes\ldots\otimes x_{\sigma_{\!k}} \!\otimes1\\[-2pt] \hspace{-23pt}\downarrow\\[-2pt]
1\otimes x_{i_1}\!\otimes x_{i_2}\!\otimes x_{\sigma_{\!3}}\!\otimes\ldots\otimes x_{\sigma_{\!k}}\!\otimes1\hspace{25pt}\end{smallmatrix}\!; i\!=\!\max\sigma_3\Big\}$ to $\mathcal{M}$,
then the unmatched vertices are $\mathcal{\mathring{M}}\!=\!\big\{ 1\!\otimes\!x_{i_1}\!\otimes\!x_{i_2}\!\otimes\!x_{i_3}\!\otimes\!x_{\sigma_{\!4}}\!\!\otimes\!\ldots\!\otimes\!x_{\sigma_{\!k}}\!\otimes\!1;\: i_1\!\leq\!i_2\!\leq\!i_3\!\leq\!\max\sigma_4\big\}$. Seeing the emerging pattern, we collectively define
\[\mathcal{M}=\Big\{\begin{smallmatrix}
1\otimes x_{i_1}\!\otimes \ldots\otimes x_{i_{r\!-\!1}}\otimes x_i\otimes x_{\sigma_{\!r}\!\setminus\!\{i\}}\otimes\ldots\otimes x_{\sigma_{\!k}} \!\otimes1\\[-2pt] \hspace{-10pt}\downarrow\\[-2pt]
1\otimes x_{i_1}\!\otimes \ldots\otimes x_{i_{r\!-\!1}}\otimes x_{\sigma_{\!r}}\!\otimes\ldots\otimes x_{\sigma_{\!k}}\!\otimes1\hspace{25pt}\end{smallmatrix}\!; r\!\geq\!1,\, i_1\!\leq\!\ldots\!\leq\!i_{r-1}\!\leq\!i\!=\!\max\sigma_r\Big\}.\]
For a multiset $\tau\!=\!\{i_1\!\leq\!\ldots\!\leq\!i_k\}$, we denote $x_{(\tau\!)}=1\!\otimes\!x_{i_1}\!\otimes\!\ldots\!\otimes\!x_{i_k}\!\otimes\!1$, and let $\overline{\tau}$ be its corresponding set, e.g. $\tau\!=\!\{1,1,2,5,5,5\}$ implies $\overline{\tau}\!=\!\{1,2,5\}$.

\begin{Prp} The $A^{\!e}$-module $A$ admits a resolution $\mathring{B}_\ast$, in which $\mathring{B}_k$ is free on symbols $\big\{x_{(\tau\!)};\, \tau\!\in\!\bbinom{\smash{[n]}}{k}\!\!\big\}$ and $\mathring{b}_k(x_{(\tau\!)})=\sum_{i\in\overline{\tau}}\big(x_i\!\otimes\!1\!+\!(-1)^k1\!\otimes\!x_i\big)x_{(\tau\!\setminus\!\{i\}\!)}$.
\end{Prp}
Since $x_i\!\otimes\!1\!\pm1\!\otimes\!x_i$ is a nonunit of $A^e$, this resolution is minimal.
\begin{proof} By AMT, it suffices to show that: (i) $\mathcal{M}$ is a Morse matching; (ii) determine its critical vertices and zig-zag paths.
\par (i) Any vertex can be written uniquely as $v\!=\! 1\!\otimes\!x_{i_1}\!\!\otimes\!\ldots\!\otimes\!x_{i_r} \!\otimes\!x_{\sigma_{\!r+1}}\!\!\otimes\!\ldots\!\otimes\!x_{\sigma_{\!k}}\!\otimes\!1$ where $i_1\!\leq\!\ldots\!\leq\!i_r$ and $r\!\geq\!0$ is maximal. Then $v$ is terminal iff $i_r\!\leq\!\max\sigma_{r+1}$, initial iff $i_r\!>\!\max\sigma_{r+1}$, and critical (=unmatched) iff $r\!=\!k$. Going up means separating the largest element from $\sigma_{r+1}$, going down means multiplying $x_{i_r}$ and $\sigma_{r+1}$; both moves are unique. This shows that $\mathcal{M}$ is a matching. The weight of an edge in $\mathcal{M}$ is $\pm1$ which is a unit of $R$. Since $A^{\otimes k+2}$ has a finite basis, the digraph $\Gamma_{B_\ast}^{\mathcal{M}}$ contains no infinite zig-zag paths with vertices in degree $k$ and $k\!-\!1$. Suppose there exists a directed cycle in $\Gamma_{B_\ast}^{\mathcal{M}}$ with vertices in degree $k$ and $k\!-\!1$, and along this cycle let \[\xymatrix@R=9.0mm@C=10mm{\ldots \ar[d]|-{\mathcal{M}}\ar[dr]& 1\!\otimes\!x_{i_1}\!\!\otimes\!\ldots\!\otimes\!x_{i_r}\!\!\otimes\!x_{\sigma_{\!r+1}}\!\!\otimes\!\ldots\!\otimes\!x_{\sigma_{\!k}}\!\!\otimes\!1 \ar[dr]|-{e} \ar[d]|-{\mathcal{M}}&\ldots\ar[d]|-{\mathcal{M}}\\
\ldots&1\!\otimes\!x_{i_1}\!\!\otimes\!\ldots\!\otimes\!x_{i_r}x_{\sigma_{\!r+1}}\!\!\otimes\!\ldots\!\otimes\!x_{\sigma_{\!k}}\!\!\otimes\!1&\ldots\vspace{-1mm}}\]
be the edge with largest $r$. What can the edge $e$ be? Multiplying $1\!\otimes\!x_{i_1}$ or $x_{\sigma_{\!k}}\!\otimes\!1$ or $x_{\sigma_{\!s}}\!\otimes\!x_{\sigma_{\!s+1}}$ for $s\!\geq\!r\!+\!2$ does not lead to a terminal vertex. Multiplying $x_{\sigma_{\!r+1}}\!\otimes\!x_{\sigma_{\!r+2}}$ and going up via $\mathcal{M}$ increases $r$, a contradiction with the maximality. Multiplying $x_{i_{\!s}}\!\otimes\!x_{i_{\!s+1}}$ and going up via $\mathcal{M}$ gives $x_{i_{\!s+1}}\!\otimes\!x_{i_{\!s}}$. Here we obtained a decrease $i_{\!s+1}\!>\!i_{\!s}$ in the sequence, but every subsequent zig-zag always puts the largest element on the left, which means that we will never regain $x_{i_{\!s}}\!\otimes\!x_{i_{\!s+1}}$.  This is a contradiction, hence $\Gamma_{B_\ast}^{\mathcal{M}}$ contains no directed cycles and $\mathcal{M}$ is a Morse matching.
\par (ii) Critical vertices are $\mathring{\mathcal{M}}\!=\! \big\{1\!\otimes\!x_{i_1}\!\!\otimes\!\ldots\!\otimes\!x_{i_k}\!\!\otimes\!1;\, i_1\!\leq\!\ldots\!\leq\!i_k \big\}\!=\!\{x_{(\tau)}\!\}$. From $x_{(\tau\!)}$, what are all the paths to all possible $x_{(\tau'\!)}$? Every zig-zag replaces $x_{i_r}\!\!\otimes\!x_{i_{r+1}}$ by $x_{i_{r+1}}\!\!\otimes\!x_{i_r}$ where $i_r\!<\!i_{r+1}$ and it has weight $(-1)^{r+1+r+1}\!=\!1$, but at the end of the zig-zag path, the sequence must be increasing. Thus after that, the only two options are: either keep moving $x_{i_r}$ to the right until it reaches $1$ and we obtain $(-1)^k1\!\otimes\!x_{i_r}\cdot x_{(\tau\setminus\{i_r\}\!)}$, or keep moving $x_{i_{r+1}}$ to the left until it reaches $1$ and we obtain $x_{i_{r+1}}\!\!\otimes\!1\cdot x_{(\tau\setminus\{i_{r+1}\}\!)}$. The sequence $x_{(\tau\!)}$ may contain constant subsequences $x_i\!\otimes\!\ldots\!\!\otimes\!x_i$, and in this case only the first $x_i$ is moved to the beginning, and only the last $x_i$ is moved to the end (because we must not multiply $x_i\!\otimes\!x_i$ since $x_i^2\!=\!0$). Thus the zig-zag paths from $x_{(\tau\!)}$ to $x_{(\tau'\!)}$ correspond to elements of the set $\overline{\tau}$.
\end{proof}\vspace{1mm}

\subsection{Homotopy equivalence}\label{hequivalence} This subsection is only used for determining the cup products on $H\!H^\ast(A;A)$ and can be skipped at first reading.

For a multiset $\tau\!=\!\{i_1\!\leq\!\ldots\!\leq\!i_k\}$, let $S_\tau$ denote the group of all its bijections, which is a subgroup of the usual symmetric group $S_k$ of all bijections of $[k]$. Given $\pi\!\in\!S_\tau$, let $x_{(\pi\tau\!)}=1\!\otimes\!x_{\pi i_1}\!\otimes\!\ldots\!\otimes\!x_{\pi i_k}\!\otimes\!1\in A^{\otimes k+2}$ denote the permuted tensor.

\begin{Prp} The homotopy equivalence $h\!:\mathring{B_\ast}\!\rightarrow\!B_\ast$ induced by $\mathcal{M}$ sends $$\textstyle{h\!:x_{(\tau\!)}\mapsto\sum_{\pi\in S_\tau}\!x_{(\pi\tau\!)}\text{ \;and\;\, } h^{\!-\!1}\!\!:x_{(\pi\tau\!)}\mapsto \icases{x_{(\tau)}}{\pi=\id}{0}{\pi\neq\id}{\!\!}{0pt}{\Big}.}$$
\end{Prp}
If $v\!=\!1\!\otimes\!x_{\sigma_1}\!\!\otimes\!\ldots\!\otimes\!x_{\sigma_k}\!\!\otimes\!1\in B_k$ is not a tensor of variables (i.e. contains at least one $x_\sigma$ with $|\sigma|\!\geq\!2$), then determining $h^{\!-\!1}(v)$ is more complicated.
\begin{proof} By AMT, $h$ sends $x_{(\tau\!)}\!\in\!\mathring{B}_k$ to $\sum_{v}\alpha_vv$ where $\alpha_v\!\in\!A^e$ is the sum of weights of all directed paths in $\Gamma_{B_\ast}^{\mathcal{M}}$ from $x_{(\tau\!)}$ to $v$, and $h^{\!-\!1}$ sends $v\!\in\!B_k$ to $\sum_{\tau}\alpha_\tau x_{(\tau\!)}$ where $\alpha_\tau\!\in\!A^e$ is the sum of weights of all directed paths in $\Gamma_{B_\ast}^{\mathcal{M}}$ from $v$ to $x_{(\tau\!)}$.
\par Determining $h^{\!-\!1}$: Since $x_{(\tau\!)}$ is not an endpoint of any $e\!\in\!\mathcal{M}$, a path from $v\!\in\!B_\ast$ to $x_{(\tau\!)}\!\in\!\mathring{B}_\ast$ must be in degrees $k$ and $k\!+\!1$. But if $v\!=\! x_{(\pi\tau')}$, then $v$ is not terminal, so a path from $v$ to $x_{(\tau\!)}$ exists iff $\tau\!=\!\tau'$ and $\pi\!=\!\id$, namely the path of length $0$.
\par Determining $h$: Since $x_{(\tau\!)}$ is not an endpoint of any $e\!\in\!\mathcal{M}$, a path from $x_{(\tau\!)}$ to $v$  must be in degrees $k$ and $k\!-\!1$. Every zig-zag out of $x_{(\tau\!)}$ switches two consecutive elements (i.e. is a transposition $\xi_i\!=\!(i,i\!+\!1)$ with weight $1$), vertex $v$ must be a permuted $x_{(\tau\!)}$. Thus it remains to prove: from any $x_{(\tau\!)}$ to any $x_{(\pi\tau\!)}$ there exists precisely one path (this is equivalent to the existence of a specific normal form on the presentation of $S_{\tau}$ by generators $\xi_i$). Notice that a path from $x_{(\tau\!)}$ to $x_{(\pi\tau\!)}$ does not correspond to an arbitrary product of transpositions, because if we switch $x_{i_r}\!\!\otimes\!x_{i_{r+1}}$, then switching $x_{i_s}\!\!\otimes\!x_{i_{s+1}}$ with $s\!\geq\!r\!+\!2$ is not possible, since it leads to a nonterminal vertex. For example, from $x_{(2,5,5,9)}$ to $x_{(9,5,5,2)}$, the only path is \[\xymatrix@R=5mm@C=2mm{
x_{(\tau\!)}\!\!=\!x_{(2,5,5,9)}\ar[d]& x_{(5,2,5,9)}\ar[d]\ar[dl]|-{\mathcal{M}}& x_{(5,5,2,9)}\ar[d]\ar[dl]|-{\mathcal{M}}& x_{(5,5,9,2)}\ar[d]\ar[dl]|-{\mathcal{M}}& x_{(5,9,5,2)}\ar[d]\ar[dl]|-{\mathcal{M}}& x_{(9,5,5,2)}\!\!=\!x_{(\pi\tau\!)}.\ar[dl]|-{\mathcal{M}}\\
x_{(25,5,9)}& x_{(5,25,9)}& x_{(5,5,29)}& x_{(5,59,2)}& x_{(59,5,2)}& \vspace{-1mm}}\]
\par In general, from $x_{(\tau\!)}\!=\!x_{(i_1,\ldots,i_k)}$ to $x_{(\pi\tau\!)}\!=\!x_{(j_1,\ldots,j_k)}$ there is the following path. First move $j_k$ to position $k$ (producing \smash{$x_{(i_1,\ldots,\!\!\widehat{~j_k},\ldots,i_k,j_k)}$}) via $\prod_{i=\epsilon j_k}^{k-1}\!\xi_i\!=:\!\xi_{(j_k)}$, where $\epsilon j_k$ is the position of the last $j_k$ in $i_1,\ldots,i_k$. Then move $j_{k\!-\!1}$ to position $k\!-\!1$ (producing \smash{$x_{(i_1,\ldots,\widehat{j_{k\!-\!1}},\ldots,\!\!\widehat{~j_k},\ldots,i_k,j_{k\!-\!1},j_k)}$}) via $\prod_{i=\epsilon j_{k\!-\!1}}^{k-2}\!\xi_i\!=:\!\xi_{(j_{k\!-\!1})}$, where $\epsilon j_{k\!-\!1}$ is the position of the last $j_{k\!-\!1}$ in $i_1,\ldots,\!\widehat{~\!j_k},\ldots,i_k$. Then move $j_{k\!-\!2}$ to position $k\!-\!2$ (producing \smash{$x_{(i_1,\ldots,\widehat{j_{k\!-\!2}},\ldots,\widehat{j_{k\!-\!1}},\ldots,\!\!\widehat{~j_k},\ldots,i_k,j_{k\!-\!2},j_{k\!-\!1},j_k)}$}) via $\prod_{i=\epsilon j_{k\!-\!2}}^{k-3}\!\xi_i\!=:\!\xi_{(j_{k\!-\!2})}$, where $\epsilon j_{k\!-\!2}$ is the position of the last $j_{k\!-\!2}$ in $i_1,\ldots,\widehat{j_{k\!-\!1}},\ldots,\!\widehat{~\!j_k},\ldots,i_k$. Continuing this process, every $j_r$ is moved to position $r$, so applying $\xi_{(j_1)}\!\cdots\xi_{(j_k)}$ to $x_{(\tau\!)}$ produces $x_{(\pi\tau\!)}$.
\par This product corresponds to a valid path, since in every $\xi_{(j_r)}$ when we move $j_r$ to the right, left of $j_r$ is an increasing sequence of $i$'s at each step. Every path corresponds to some $\xi_{(j_1)}\!\cdots\xi_{(j_k)}$, since the first step must be moving $j_k$ to position $k$ (if we first apply $\xi_r$ with $r\!<\!\epsilon j_k$, then moving $j_k$ to the right does not lead to a terminal vertex; if we first apply $\xi_r$ with $r\!>\!\epsilon j_k$, then $j_k$ will not be able to move to the right across the transposed $i_{r+1},i_r$), and then inductively every next step must be moving $j_r$ to position $r$. Any permutation $\pi\!\in\!S_\tau$ is equal to a unique product $\xi_{(j_1)}\!\cdots\xi_{(j_k)}$, since applying $\xi_{(j_1)}\!\cdots\xi_{(j_k)}$ to $x_{(\tau\!)}$ produces $x_{(j_1,\ldots,j_k)}$.
\end{proof}

\vspace{2mm}
\section{Homology} In the computation below, we implicitly use the isomorphisms from \ref{identifications}.

\begin{Thm} If $R\!=\!\Z$, then $H\!H_k(A;A)\cong \Z^F\!\!\oplus\!\Z_2^{\,T}$\!, where
\[\textstyle{F=2^{n\!-\!1}\!\bbinom{n}{k}+\icases{1}{k=0}{0}{k\geq1}{\!\!}{1pt}{\Big} \text{ \;\;and\;\; }T=(-1)^{k+1}+2^{n\!-\!1}\!\sum_{0\leq i\leq k}(-1)^{k-i}\!\bbinom{n}{i}.}\vspace{-0mm}\]
In particular, if $R\!=\!K$ is a field, then\:\vspace{-2mm} $$\dim_K H\!H_k(A;A) = \iicases{2^n\bbinom{n}{k}}{\chr K=2}{2^{n\!-\!1}\!\bbinom{n}{k}}{\chr K\neq2,\:k\geq1}{2^{n\!-\!1}+1}{\chr K\neq2,\:k=0}{}{-1pt}{\Bigg}.$$
\end{Thm}
\begin{proof} 
$A\!\otimes_{\!A^{\!e}}\!\!\mathring{B}_k$ has an $R$-module basis $\{x_\sigma\!\otimes\!x_{(\tau\!)};\, \sigma\!\subseteq\![n], \tau\!\in\!\bbinom{\smash{[n]}}{k}\}$. The homotopy equivalence $B_\ast \!\simeq\! \mathring{B}_\ast$ implies $A\!\otimes_{\!A^{\!e}}\!\!B_\ast \!\simeq\! A\!\otimes_{\!A^{\!e}}\!\!\mathring{B}_\ast \!=:\! \mathring{C}_\ast$ in which the new boundary is \[\begin{array}{c}\mathring{\partial}_k(x_\sigma\!\otimes\!x_{(\tau\!)})=x_\sigma\!\otimes\!\mathring{b}_k(x_{(\tau\!)})=
x_\sigma\!\otimes\! \sum_{i\in\overline{\tau}}\big(x_i\!\otimes\!1\!+\!(-1)^k1\!\otimes\!x_i\big)x_{(\tau\!\setminus\!\{i\}\!)}=\\
\sum_{i\in\overline{\tau}}\big(x_\sigma x_i\!+\!(-1)^kx_ix_\sigma\big)\!\otimes\! x_{(\tau\!\setminus\!\{i\}\!)}=
\sum_{i\in\overline{\tau}}\big((-1)^{|\sigma|}\!+\!(-1)^{|\tau|}\big)x_ix_\sigma\!\otimes\! x_{(\tau\!\setminus\!\{i\}\!)}.\end{array}\]
The chain complex $(\mathring{C}_\ast,\mathring{\partial}_\ast)$ is a direct sum of subcomplexes \[\begin{array}{l}
 C'_\ast\!=\!\big\langle x_\sigma\!\otimes\!x_{(\tau)};\, (-1)^{|\sigma|}\!=\!(-1)^{|\tau|}\big\rangle,\: \partial'(x_\sigma\!\otimes\!x_{(\tau\!)})= \!\!\!\! \underset{i\in\overline{\tau}\setminus\sigma}{\sum}\!\!(-1)^{|\sigma|+j}2\,x_{\sigma\cup\{i\}}\!\otimes\!x_{(\tau\!\setminus\!\{i\}\!)},\\[-3pt]
C''_\ast\!=\!\big\langle x_\sigma\!\otimes\!x_{(\tau)};\, (-1)^{|\sigma|}\!\neq\!(-1)^{|\tau|}\big\rangle,\: \partial''\!=\!0,\text{ where } x_ix_\sigma\!=\!(-1)^jx_{\sigma\cup\{i\}}.\end{array}\]
\par Let $R\!=\!K$ be a field. If $\chr\,K\!=\!2$, then $\mathring{\partial}_\ast\!=\!0$ and thus $\dim_K H\!H_k(A;A)=|\{x_\sigma\!\otimes\!x_{(\tau\!)}\}|= 2^n\!\bbinom{n}{k}$. If $\chr\,K\!\neq\!2$, then $(-1)^{|\sigma|}2$ is a unit of $K$\!, so we define \[\mathcal{M}'=\Big\{\begin{smallmatrix}
x_\sigma\otimes x_{(\tau\cup\{i\}\!)}\\[-2pt] \hspace{-10pt}\downarrow\\[-2pt]
x_{\sigma\cup\{i\}}\otimes x_{(\tau\!)}\end{smallmatrix}\!;\; \max\sigma<i\geq\max\tau\Big\}.\] This is a Morse matching on $C'_\ast$ with no zig-zags, and the critical vertices are $\mathring{\mathcal{M}}'\!=\!\{x_\emptyset\!\otimes\!x_{(\emptyset)}\}$. Hence $C'_\ast$ and $C''_\ast$ contribute $\icases{1}{k=0}{0}{k\geq1}{\!\!}{1pt}{\Big}$ and $2^{n\!-\!1}\!\bbinom{n}{k}$ to $\dim_K\!H\!H_k(A;A)$.
\par Let $R\!=\!\Z$, so $\pm2$ is no longer a unit. Instead of studying $(C'_\ast,\partial')$, let us observe $(C'_\ast,\frac{\partial'}{2})$. We will inductively use a trick: if matrices $\Z^l\!\overset{\alpha}{\longleftarrow}\!\Z^m\!\overset{\beta}{\longleftarrow}\!\Z^n$ satisfy $\alpha\beta\!=\!0$ and have SNF (Smith normal forms) $\alpha\!\equiv\! \mathrm{diag}(a_1,\ldots,a_r,0,\ldots,0)$ and $\beta\!\equiv\! \mathrm{diag}(b_1,\ldots,b_s,0,\ldots,0)$, then $\frac{\Ker\,\alpha}{\Im\,\beta}\cong \Z^{m-r-s}\!\oplus\!\Z_{b_1}\!\oplus\!\ldots\!\oplus\!\Z_{b_s}$ and $\frac{\Ker\,2\alpha}{\Im\,2\beta}\cong \Z^{m-r-s}\!\oplus\!\Z_{2b_1}\!\oplus\!\ldots\!\oplus\!\Z_{2b_s}$. Since the above $\mathcal{M}'$ is a Morse matching on $(C'_\ast,\!\frac{\partial'_\ast}{2})$, it follows that $H_k(C'_\ast,\!\frac{\partial'_\ast}{2})\cong \icases{\Z}{k=0}{0}{k\geq1}{\!\!}{1pt}{\Big}$. Hence boundaries $\frac{\partial'_k}{2}$ have SNFs\vspace{-3mm} \[\Z^{r_0}\!\!\xleftarrow{\!\mathrm{diag}(\overbrace{\scriptstyle1,\ldots,1}^{\smash{r_0-1}},0,\ldots,0)\!\!}\!
\Z^{r_1}\!\!\xleftarrow{\!\!\mathrm{diag}(\!\!\!\smash{\overbrace{\scriptstyle1,\ldots,1}^{r_1\!-r_0+1}}\!\!\!,0,\ldots,0)\!\!}\!\Z^{r_2}\!\!\longleftarrow\!\cdots\!\longleftarrow\!
\Z^{r_k}\!\!\xleftarrow{\!\!\mathrm{diag}(\!\!\!\!\!\!\!\!\!\!\overbrace{\scriptstyle1,\ldots,1}^{\sum_{i=\!-\!1}^k\!(\!-\!1\!)^{k\!-\!i}r_i}\!\!\!\!\!\!\!\!\!\!,0,\ldots,0)\!\!} \!\Z^{r_{k+1}}\!\!\longleftarrow\!\cdots\!,\vspace{-1mm}\] where $r_k\!=\!2^{n\!-\!1}\!\!\bbinom{n}{k}$ and $r_{-1}\!=\!1$. Then we conclude that $H_k(C'_\ast,\partial'_\ast)\cong \Z^F\!\oplus\!\Z_2^T$, where free rank is $F\!=\!\icases{1}{\!k=0}{0}{\!k\geq1}{\!\!}{-1pt}{\big}$ and $2$-torsion rank is $\sum_{i=\!-\!1}^k\!(\!-\!1\!)^{k\!-\!i}r_i$.
\end{proof}

\vspace{2mm}
\section{Cohomology} Using finite additivity of $\Hom(-,-)$, we see that $\Hom_{A^{\!e}}(B_k,A)$ has a basis $\big\{\varphi_{v\!,\sigma}; v\!\in\!\{1\!\otimes\!x_{\!\sigma_{\!1}}\!\!\otimes\!\ldots\!\otimes\!x_{\!\sigma_{\!k}}\!\otimes\!1; \sigma_1,\ldots,\sigma_k\!\subseteq\![n]\}, \sigma\!\subseteq\![n]\big\}$ and $\Hom_{A^{\!e}}(\mathring{B}_k,A)$ has a basis $\big\{\varphi_{\tau\!,\sigma}; \tau\!\in\!\!\bbinom{\smash{[n]}}{k}\!\!, \sigma\!\subseteq\![n]\big\}$, the dual bases of $B_k$ and $\mathring{B}_k$, where $\varphi_{v\!,\sigma}(v')\!=\! \icases{x_\sigma}{v=v'}{0}{v\neq v'}{\!\!}{-1pt}{\Big}$.

\begin{Thm} If $R\!=\!\Z$, then $H\!H^k(A;A)\cong \Z^F\!\!\oplus\!\Z_2^{\,T}$, where
\[\textstyle{F\!=\!2^{n\!-\!1}\!\!\bbinom{n}{k}\!+\!\icases{1}{k=0\text{ and }n\text{ \!odd}}{0}{\text{otherwise}}{\!\!}{1pt}{\Big} \text{ \;and\; }
T\!=\!2^{n\!-\!1}\!\sum_{0\leq i<k}(-1)^{k-\!1\!-i}\!\!\bbinom{n}{i}\!+\!\icases{(\!-\!1\!)^k}{n\text{ \!odd}}{0}{n\text{ \!even}}{\!\!}{1pt}{\Big}.}\vspace{-0mm}\]
In particular, if $R\!=\!K$ is a field, then\:\vspace{-2mm} $$\dim_K H\!H^k(A;A) = \iicases{2^n\bbinom{n}{k}}{\chr K=2}{2^{n\!-\!1}\!\bbinom{n}{k}}{\chr K\neq2,\:(k\geq1\text{ or }n\text{ \!even})}{2^{n\!-\!1}+1}{\chr K\neq2,\:(k=0\text{ and }n\text{ \!odd})}{}{-1pt}{\Bigg}.$$
\end{Thm}
\begin{proof} Homotopy equivalence $B_\ast \!\simeq\! \mathring{B}_\ast$ implies $\Hom_{\!A^{\!e}}\!(B_\ast,A) \!\simeq\! \Hom_{\!A^{\!e}}\!(\mathring{B}_\ast,A) \!=:\! \mathring{C}^\ast$\!. In this complex, the new coboundary is \[\begin{array}{c}\big(\mathring{\delta}_k\,\varphi_{\tau\!,\sigma}\big)(x_{(\tau'\!)})= \varphi_{\tau\!,\sigma}\,\big(\mathring{b}_k(x_{(\tau'\!)})\big)= \sum_{i\in\overline{\tau}'}(x_i\!\otimes\!1\!+\!(-1)^{|\tau'|}1\!\otimes\!x_i)\varphi_{\tau\!,\sigma}\big(x_{(\tau'\!\setminus\!\{i\}\!)}\big)=\\[3pt]
\sum_{i\in\overline{\tau}'}(x_i\!\otimes\!1\!+\!(-1)^{|\tau'|}1\!\otimes\!x_i)\!\cdot\!\icases{x_\sigma}{\tau'=\tau\cup\{i\}}{0}{\tau'\neq\tau\cup\{i\}}{\!\!}{1pt}{\Big}=
\sum_{i\in\overline{\tau}'}\!\icases{((-1)^{|\sigma|}+(-1)^{|\tau'|})x_\sigma x_i}{\tau'=\tau\cup\{i\}}{0}{\tau'\neq\tau\cup\{i\}}{\!\!}{1pt}{\Big},\\[3pt]\text{ therefore }\mathring{\delta}_k(\varphi_{\tau\!,\sigma})=\sum_{i\in[n]\setminus\sigma}\big((-1)^{|\sigma|}\!-\!(-1)^{|\tau|}\big)(-1)^j\varphi_{\tau\cup\{i\}\!,\sigma\cup\{i\}},\end{array}\]
where $j$ is the number of transpositions, needed to transform $x_\sigma x_i$ into $x_{\sigma\cup\{i\}}$. The cochain complex $(\mathring{C}^\ast,\mathring{\delta}^\ast)$ is a direct sum of subcomplexes \[\begin{array}{l}
 C'^\ast=\!\big\langle \varphi_{\tau\!,\sigma};\, (-1)^{|\sigma|}\!\neq\!(-1)^{|\tau|}\big\rangle,\: \delta'(\varphi_{\tau\!,\sigma})=\!\!\!\! \underset{i\in[n]\setminus\sigma}{\sum}\!\!\!(-1)^{|\sigma|+j}2\,\varphi_{\tau\cup\{i\}\!,\sigma\cup\{i\}},\\[-6pt]
C''^\ast\!=\!\big\langle \varphi_{\tau\!,\sigma};\, (-1)^{|\sigma|}\!=\!(-1)^{|\tau|}\big\rangle,\: \delta''\!=\!0.\end{array}\]
\par Let $R\!=\!K$ be a field. If $\chr\,K\!=\!2$, then $\mathring{\delta}^\ast\!=\!0$ and thus $\dim_K H\!H^k(A;A)=|\{\varphi_{\tau\!,\sigma}\}|= 2^n\!\bbinom{n}{k}$. If $\chr\,K\!\neq\!2$, then $(-1)^{|\sigma|+j}2$ is a unit of $K$\!, so we define \[\mathcal{M}'=\Big\{\begin{smallmatrix}
\varphi_{\tau\!\cup\!\{i\}\!,\sigma\!\cup\!\{i\}}\\[-1pt] \hspace{-0pt}\downarrow\\[-1pt]
\varphi_{\tau\!,\sigma}\end{smallmatrix}\!;\; [n]\!\setminus\!\overline{\tau}\!\supseteq\![i\!-\!1]\!\subseteq\!\sigma,\, i\!\notin\!\sigma\Big\}.\] This is a Morse matching on $C'^\ast$ with  critical vertices $\mathring{\mathcal{M}}'\!=\!\{\varphi_{\emptyset,[n]};\, n\text{ \!odd}\}$. Hence $C'^\ast$ and $C''^\ast$ contribute $\icases{1}{k=0\text{ and }n\text{ \!odd}}{0}{k\geq1\text{ or }n\text{ \!even}}{\!\!}{1pt}{\Big}$ and $2^{n\!-\!1}\!\!\bbinom{n}{k}$ to $\dim_K\!H\!H^k(A;A)$.
\par Let $R\!=\!\Z$, so $\pm2$ is no longer a unit. Since $\mathcal{M}'$ is a Morse matching on $(C'^\ast,\!\frac{\delta'^\ast}{2})$, it follows that $H^k(C'^\ast,\!\frac{\delta'^\ast}{2})\!\cong\! \icases{\Z}{k=0\text{ and }n\text{ \!odd}}{0}{k\geq1\text{ or }n\text{ \!even}}{\!\!}{1pt}{\Big}$. Hence coboundaries $\frac{\delta'^k}{2}$ have SNFs\vspace{-3mm} \[\Z^{r_0}\!\!\xrightarrow{\!\!\mathrm{diag}(\overbrace{\scriptstyle1,\ldots,1}^{\smash{r_0-1}},0,\ldots,0)\!\!}\!
\Z^{r_1}\!\!\xrightarrow{\!\!\mathrm{diag}(\!\!\!\smash{\overbrace{\scriptstyle1,\ldots,1}^{r_1\!-r_0+1}}\!\!\!,0,\ldots,0)\!\!}\!\Z^{r_2}\!\!\longrightarrow\!\cdots\!\longrightarrow\!
\Z^{r_k}\!\!\xrightarrow{\!\!\mathrm{diag}(\!\!\!\!\!\!\!\!\!\!\overbrace{\scriptstyle1,\ldots,1}^{\sum_{i=\!-\!1}^k\!(\!-\!1\!)^{k\!-\!i}r_i}\!\!\!\!\!\!\!\!\!\!,0,\ldots,0)\!\!} \!\Z^{r_{k+1}}\!\!\text{ if }n\!\notin\!2\N,\vspace{-2mm}\]
\[\Z^{r_0}\!\!\xrightarrow{\!\!\mathrm{diag}(\overbrace{\scriptstyle1,\ldots,1}^{\smash{r_0}},0,\ldots,0)\!\!}\!
\Z^{r_1}\!\!\xrightarrow{\!\!\mathrm{diag}(\smash{\overbrace{\scriptstyle1,\ldots,1}^{r_1\!-r_0}},0,\ldots,0)\!\!}\!\Z^{r_2}\!\!\longrightarrow\!\cdots\!\longrightarrow\!
\Z^{r_k}\!\!\xrightarrow{\!\!\mathrm{diag}(\!\!\!\!\!\!\!\!\!\overbrace{\scriptstyle1,\ldots,1}^{\sum_{i=0}^k\!(\!-\!1\!)^{k\!-\!i}r_i}\!\!\!\!\!\!\!\!\!,0,\ldots,0)\!\!} \!\Z^{r_{k+1}}\!\!\text{ if }n\!\in\!2\N.\vspace{1mm}\] Then we deduce that $H^k(C'^\ast\!,\delta'^\ast) \cong \Z^F\!\!\oplus\!\Z_2^T$, where free rank is $F\!=\!\icases{1}{\!k=0\text{ and }\!n\notin2\N}{0}{\!k\geq1\text{ or }n\in2\N}{\!\!}{2pt}{\Big}\!$ and $2$-torsion rank is $T\!=\!\icases{(\!-\!1\!)\smash{^{\!k}}\!}{\!n\notin2\N}{0\!}{\!n\in2\N}{\!\!}{1pt}{\Big}\!+\!\sum_{i=0}^{k\!-\!1}(\!-\!1\!)^{k-1-i}r_i$, with $r_k\!=\!2^{n\!-\!1}\!\!\bbinom{n}{k}$.
\end{proof}

\vspace{2mm}
\section{Cup product} For any associative unital algebra $A$, the cup product of $f\!\in\!\Hom_{\!R}(A^{\otimes k}\!,A)$ and $g\!\in\!\Hom_{\!R}(A^{\otimes l}\!,A)$ is $f\!\smile\!g\!\in\! \Hom_{\!R}(A^{\otimes k+l}\!,A)$ given by\vspace{-1mm} \[a_1\!\otimes\!\ldots\!\otimes\!a_k\!\otimes\!a_{k+1}\!\otimes\!\ldots\!\otimes\!a_{k+l} \!\longmapsto\! f(a_1\!\otimes\!\ldots\!\otimes\!a_k)\!\cdot\!g(a_{k+1}\!\otimes\!\ldots\!\otimes\!a_{k+l})\vspace{-1mm}\] and makes $H\!H^\ast(A;\!A)$ a graded-commutative $R$-algebra.\vspace{1mm}

\begin{Thm} Let $R\!=\!K$ be a field. If $\chr\,K\!=\!2$, then\vspace{-1mm} \[H\!H^\ast(A;A)\cong \Lambda[x_1,\ldots,x_n]\otimes_{\!K} \!K[x_1,\ldots,x_n],\vspace{-1mm}\] where $x_\sigma\!\!\otimes\!x_\tau$\! corresponds to $\varphi_{\tau\!,\sigma} \!\in\! \Hom_R(\!A^{\!\otimes k}\!,\!A)$. If $\chr\,K\!\neq\!2$, then $H\!H^\ast(A;A)$ is isomorphic to the subalgebra of $\Lambda[x_1,\ldots,x_n]\otimes_{\!K} \!K[x_1,\ldots,x_n]$, spanned by\vspace{-1mm}
\[\big\{x_\sigma\!\!\otimes\!x_{\tau};\: (-1)^{|\sigma|}\!=\!(-1)^{|\tau|}\!\big\} \cup \big\{x_{[n]}\!\otimes\!1\big\}.\]
\end{Thm}\vspace{1mm}
Thus if $\chr\,K\!\neq\!2$, then $H\!H^\ast(A;A)$ has algebra generators
\[\big\{1\!\otimes\!x_{\{i,j\}};\: i\!\leq\!j\!\in\![n]\big\} \cup
  \big\{x_{\{i,j\}}\!\otimes\!1;\: i\!<\!j\!\in\![n]\big\} \cup
  \big\{x_i\!\otimes\!x_j;\: i,j\!\in\![n]\big\} \cup
  \big\{x_{[n]}\!\otimes\!1\big\},\]
and is subject to the relations \cite[Table 1]{citearticleHanXuHHEA}, though their result is missing $x_{[n]}\!\otimes\!1$.
\begin{proof} \ref{hequivalence} induces a homotopy equivalence  $\overline{h}\!: \Hom_{\!R}(B_\ast,A) \!\longrightarrow\! \Hom_{\!R}(\mathring{B}_\ast,A)$, $\varphi\!\mapsto\!\varphi\!\circ\!h$, given by
$\overline{h}(\varphi_{\pi\tau\!,\sigma})\!=\! \varphi_{\tau\!,\sigma}$, $\overline{h}(\varphi_{v\!,\sigma})\!=\! 0$, $\smash{\overline{h}^{-\!1}}\!(\varphi_{\tau\!,\sigma})\!=\! \varphi_{\tau\!,\sigma}\!+\!\sum_v\!\alpha_v\varphi_{v\!,\sigma}$, where $v$ is not a tensor of variables. Indeed, $\varphi_{\pi\tau\!,\sigma}\!\circ\!h(x_{(\tau')})\!=\! \varphi_{\pi\tau\!,\sigma}\!\sum_{\pi'\in S_{\tau'}}\!x_{(\pi'\tau')}\!=\! \icases{x_\sigma}{\tau=\tau'}{0}{\tau\neq\tau'}{\!\!}{-2pt}{\big}$, $\varphi_{v\!,\sigma}\!\circ\!h(x_{(\tau\!)})\!=\! \varphi_{v\!,\sigma}\!\sum_{\pi\in S_\tau}\!x_{(\pi\tau)}\!=\! 0$, $\varphi_{\tau\!,\sigma}\!\circ\!h^{\!-\!1}(x_{(\pi\tau')})\!=\! \icases{x_\sigma}{\tau=\tau'\text{ \!and\! }\pi=\id}{0}{\tau\neq\tau'\text{ or }\pi\neq\id}{\!\!}{-1pt}{\Big}$. Then \vspace{-1mm}
\[\begin{array}{c}\varphi_{\tau\!,\sigma} \!\!\smile\! \varphi_{\tau'\!\!,\sigma'}=
\overline{h}\big(\overline{h}^{-\!1}\!(\varphi_{\tau\!,\sigma}) \!\smile\! \overline{h}^{-\!1}\!(\varphi_{\tau'\!\!,\sigma'})\big)=
\overline{h}\big((\varphi_{\tau\!,\sigma}\!+\!\ldots) \!\smile\! (\varphi_{\tau'\!,\sigma'}\!+\!\ldots)\big)=\\[1mm]
\overline{h}\big(\varphi_{\tau\!,\sigma}\!\!\smile\!\varphi_{\tau'\!,\sigma'}\!+\!\ldots\big)=
\overline{h}\big((-1)^j\varphi_{\tau\otimes\tau'\!\!,\sigma\sqcup\sigma'}\!+\!\ldots\big)= (-1)^j\varphi_{\tau\cup\tau'\!,\sigma\sqcup\sigma'},
\end{array}\]
where $(-1)^j$ is the sign of $x_\sigma x_{\sigma'}$. The three dots represent summands of the form $\varphi_{v\!,\sigma}\!\!\smile\!\varphi_{v'\!\sigma'}\!=\! (-1)^j\varphi_{v\otimes v'\!,\sigma\sqcup\sigma'}$ where either $v$ or $v'$ is not a tensor of variables.
\end{proof}\vspace{2mm}

\begin{Rmk} In general, $H\!H_\ast(A;\!A)$ and $H\!H^\ast(A;\!A)$ are modules over the center $Z(A)$. The multiplication map $A\!\otimes\!A\!\longrightarrow\!A, x\!\otimes\!y\!\mapsto\!xy$ is an algebra morphism iff $A$ is commutative, and then by \cite[4.2.6]{citeLodayCH} it induces a \emph{shuffle} product on $H\!H_\ast(A;\!A)$, given by $(a\!\otimes\!a_1\!\otimes\!\ldots\!\otimes\!a_i)\cdot(a'\!\otimes\!a_{i+1}\!\otimes\!\ldots\!\otimes\!a_{i+j})= $\vspace{-0mm}
\[\textstyle{\sum_{\pi\in S_{i+j},\,\pi_1<\ldots<\pi_i,\,\pi_{i+1}<\ldots<\pi_{i+j}}\! \mathrm{sgn}\pi\,aa'\!\otimes\! a_{\pi^{\!-\!1}_1}\!\!\otimes\!\ldots\!\otimes\!a_{\pi^{\!-\!1}_i}\!\!\otimes\!a_{\pi^{\!-\!1}_{i+1}}\!\!\otimes\!\ldots\!\otimes\!a_{\pi^{\!-\!1}_{i+j}}},\vspace{-0mm}\]
i.e. $a_1\!\otimes\!\ldots\!\otimes\!a_i$ and $a_{i+1}\!\otimes\!\ldots\!\otimes\!a_{i+j}$ are subsequences of $a_{\pi^{\!-\!1}_1}\!\!\otimes\!\ldots\!\otimes\!a_{\pi^{\!-\!1}_{i+j}}$.
\par In our case of an exterior algebra, let $\chr\,K\!=\!2$, so that $A$ is commutative. Then $H\!H_\ast(A;\!A)$ and $H\!H^\ast(A;\!A)$ are commutative graded $A$-algebras. Since all above isomorphisms are $A$-linear, a similar argument using \ref{hequivalence} shows that $$H\!H_\ast(A;\!A)\cong H\!H^\ast(A;\!A)\cong A[y_1,\ldots,y_n].$$
\end{Rmk}

\vspace{2mm}
\section{Afterword} Given a simplicial complex $\Delta$ with vertex set $[n]$, the Stanley-Reisner algebra (SRa) of $\Delta$ is the monomial quotient $\Lambda[x_1,\ldots,x_n\,|\,x_\sigma;\,\sigma\!\notin\!\Delta]$. Its module basis consists of all the simplices of $\Delta$. In the above statements, we have computed HH of the SRa of an $n$-ball. It is an interesting task to determine HH of the SRa of a sphere $\Lambda[x_1,\ldots,x_n\,|\,x_1\cdots x_n]$ or a path $\Lambda[x_1,\ldots,x_n\,|\,x_ix_j;\, |i\!-\!j|\!\geq\!2]$ or a cycle $\Lambda[x_1,\ldots,x_n\,|\,x_ix_j;\, |i\!-\!j\!\mod\!n|\!\geq\!2]$ or other more general simplicial complexes. This still appears to be an open problem.

\end{document}